\documentclass[12pt]{amsart}
\usepackage{amsfonts, amsmath, amsthm}

\usepackage{graphicx}
\usepackage{epstopdf}
\usepackage{psfrag}
\newtheorem{pro}{Proposition}[section]
\newtheorem{thm}[pro]{Theorem}
\newtheorem{lem}[pro]{Lemma}

\newtheorem{cor}[pro]{Corollary}

\theoremstyle{definition}
\newtheorem{dfn}[pro]{Definition}
\newtheorem{ex}[pro]{Examples}
\newtheorem{rmk}[pro]{Remark}

\theoremstyle{remark}

\newcommand{\bdy}{\partial}

\newcommand{\thick}[1]{{\rm Thick}(#1)}
\newcommand{\thin}[1]{{\rm Thin}(#1)}

\title[Heegaard structure respects JSJ decompositions]{Heegaard structure respects complicated JSJ decompositions}

\date{\today}
\address{Pitzer College}
\email{bachman@pitzer.edu}
\author{David Bachman}
\begin{document}

\address{Quest University Canada}
\email{rdt@questu.ca}
\author{Ryan Derby-Talbot}

\address{DePaul University}
\email{esedgwick@cdm.depaul.edu}
\author{Eric Sedgwick}

\begin{abstract}
Let $M$ be a 3-manifold with torus boundary components $T_1$ and $T_2$. Let $\phi \colon T_1 \to T_2$ be a homeomorphism, $M_\phi$ the manifold obtained from $M$ by gluing $T_1$ to $T_2$ via the map $\phi$, and $T$ the image of $T_1$ in $M_\phi$. We show that if $\phi$ is ``sufficiently complicated" then any incompressible or strongly irreducible surface in $M_\phi$ can be isotoped to be disjoint from $T$. It follows that every Heegaard splitting of a 3-manifold admitting a ``sufficiently complicated'' JSJ decomposition is an amalgamation of Heegaard splittings of the components of the JSJ decomposition.
\end{abstract}

\maketitle

\section{Introduction}
\label{s:intro}

Let $M$ be a (possibly disconnected) 3-manifold with homeomorphic boundary components $F_1$ and $F_2$. Let $\phi \colon F_1 \to F_2$ be a homeomorphism, and let $M_\phi$ be the manifold obtained from $M$ by gluing $F_1$ to $F_2$ via the map $\phi$. Finally, let $F$ denote the image of $F_1$ in $M_\phi$. So-called incompressible and strongly irreducible surfaces have been seen to be vital to the study of 3-manifolds. A natural question is how such surfaces in $M_\phi$ are related to surfaces in $M$. In many cases it is known that with some assumption on the map $\phi$, a class of such surfaces in $M_{\phi}$ can be isotoped to be disjoint from $F$. In other words, each surface in the class is isotopic into the original manifold $M$. When this conclusion follows, then we say $F$ is a {\it barrier} for the class.

If $F$ is separating and its genus is greater than one, then Lackenby \cite{Lackenby2004} and Souto \cite{Souto} have independently shown that if $\phi$ is ``sufficiently complicated" (in different contexts), then $F$ is a barrier for all incompressible and strongly irreducible surfaces of bounded genus. Li \cite{Li} also obtained a similar result, including the case that $F$ is a (possibly non-separating) torus. In some cases it is even known that $F$ is a barrier for all incompressible and strongly irreducible surfaces {\it regardless of their genus}. For example, this follows from \cite{Bachman2006} when $F$ is a separating torus in $M_\phi$.

In this paper we look at the situation when the gluing surface is a torus $T$ in $M_\phi$ that may be non-separating. Our main result is that the same conclusion holds: when $\phi$ is ``sufficiently complicated" the surface $T$ is a barrier for all incompressible and strongly irreducible surfaces, regardless of their genus. A consequence of this is that all Heegaard splittings of $M_\phi$ arise from Heegaard splittings of $M$ in a natural way.

To prove our main result, the first challenge is to construct a reasonable definition of the term ``sufficiently complicated" that depends only on the manifold $M$. This is accomplished in Section \ref{s:Normal-like}, where we construct a complexity for gluing maps between two torus boundary components of a 3-manifold called the {\it $c$-distance}, where $c$ depends on the class of surfaces under consideration. We then establish our first theorem:

\begin{thm}
\label{main_thm1}
Let $M$ be a compact, orientable, irreducible 3-manifold with incompressible boundary and let $T_1$ and $T_2$  be two torus components of $\partial M$. Let $\phi \colon T_1 \to T_2$ be a gluing map whose 1-distance is at least two. Then every closed, orientable, incompressible and strongly irreducible surface in $M_\phi$ can be isotoped to be disjoint from the image of $T_1$ in $M_\phi$.
\end{thm}

That is, a sufficiently complicated gluing map between torus boundary components of a 3-manifold creates a barrier to all incompressible and strongly irreducible surfaces. The most interesting case of Theorem \ref{main_thm1} is when $M$ is connected, since similar results were previously obtained in \cite{Bachman2006} for the case that $M$ is disconnected. In the special case when $M \cong T^2 \times I$ this follows from \cite{Cooper1999}.

The proof of Theorem~\ref{main_thm1} is given in Section~\ref{s:IntersectingWithTori}, considering the incompressible and strongly irreducible cases separately. In Section~\ref{s:selfamalgamation} we review the definitions of {\it generalized Heegaard splitting} and {\it amalgation}, and then extend the definition of $c$-distance to gluings along arbitrarily many tori. This allows us to put everything together to prove the following theorem about Heegaard splittings of certain 3-manifolds with non-trivial JSJ decompositions (see Section~\ref{s:selfamalgamation} for relevant definitions):

\begin{thm}
\label{JSJtheorem}
Let $M$ be a closed, irreducible, orientable 3-manifold that admits a JSJ decomposition whose collection of JSJ tori has $1$-distance at least 2. Then every Heegaard splitting of $M$ is an amalgamation of Heegaard splittings of the components of the JSJ decomposition of $M$ over the gluing maps.
\end{thm}

Theorem~\ref{JSJtheorem} greatly simplifies the Heegaard structure of 3-manifolds that are ``sufficiently complicated'' in the above sense. For example, despite being toroidal, such manifolds admit only finitely many Heegaard splittings of any given genus up to isotopy (see Corollary~\ref{c:finiteness}). 

In the proof of Theorem \ref{main_thm1} we show that $\phi$ need only to have $1$-distance one to create a barrier to incompressible surfaces. In Section \ref{s:MainTheoremSection} we address the question of the existence of $c$-distance one gluing maps. There we show the following:

\begin{thm}
\label{main_thm2}
Let $M$ be a compact, orientable, irreducible 3-manifold with incompressible boundary and let $T_1$ and $T_2$  be two torus components of $\partial M$. Let $\psi \colon T_1 \to T_2$ and $\sigma \colon T_2 \to T_2$ be homeomorphisms, where $\sigma$ is Anosov. Then for each $c$ there exists a positive integer $N$ such that for $n \geq N$, the map $\sigma^n \psi$ has $c$-distance at least one.
\end{thm}

As above, the most interesting case is when $M$ is connected since otherwise this result is implied by \cite{Hatcher}. Theorems \ref{main_thm1} and \ref{main_thm2} imply that the composition of any gluing map between two torus boundary components of some 3-manifold with a sufficiently high power of an Anosov map creates a barrier to all incompressible surfaces.  We expect a similar result holds for strongly irreducible surfaces as well.

\section{Boundary slopes of surfaces that meet the 2-skeleton normally}
\label{s:Normal-like}

Let $M$ be a compact, orientable, irreducible, triangulated 3-manifold consisting of one or two components\footnote{The results for this section in the case that $M$ has two components have been obtained previously in \cite{Hatcher} and \cite{Jaco2003}. While our focus is therefore on the case that $M$ is connected, the results we establish here are general and include these previously obtained results in our own terminology, which we provide for completeness and ease of exposition in later sections.}. Presently we shall define what it means for a surface to meet the 2-skeleton of the triangulation normally. Such surfaces can then be broken up into {\it compatibility classes} with certain properties (see Definitions~\ref{d:compatibilityclass} and \ref{d:typed}). Suppose $T_1$ and $T_2$ are torus components of $\bdy M$. The main result of this section is an extension of Hatcher's theorem \cite{Hatcher}, \cite{Floyd-Oertel} from slopes of essential surfaces to slopes of surfaces which meet the 2-skeleton of a triangulation normally.  Namely, we show that each compatibility class $\mathcal C$ determines an element $\Phi_{\mathcal C}$ of $SL_2(\mathbb Q)$ such that for each $S \in \mathcal C$, $\Phi_{\mathcal C}(\langle S \cap T_2\rangle)= \langle S \cap T_1\rangle$, where $\langle S \cap T_i \rangle$ is the slope of $S \cap T_i$. These maps will allow us to define the complexity for gluing maps $\phi:T_1 \to T_2$ referred to in Section~\ref{s:intro}.

Much of the terminology for this section is based on that of Jaco and Sedgwick \cite{Jaco2003}. Also, the reader is referred to \cite{JacoTollefson} for a more elementary introduction to normal surface theory.

Fix a triangulation $\mathcal T$ of $M$ that restricts to one-vertex triangulations on $T_1$ and $T_2$. Such triangulations exist by \cite{Jaco2003}.

\begin{dfn}
A properly embedded arc in a 2-simplex of $\mathcal T$ is {\it normal} if it connects distinct 1-simplices. A loop on the boundary of a tetrahedron is normal if it consists of normal arcs. A {\it normal isotopy} of such a loop is an isotopy that restricts to an isotopy in each boundary simplex of the tetrahedron.  It is well known that each normal loop consists of 3 or $4n$ normal arcs. The {\it complexity} of a normal loop which consists of 3 arcs is zero, and of a normal loop that consists of $4n$ arcs is $n-1$.
\end{dfn}

\begin{dfn}
A surface $S$ in $M$ is {\em normal with respect to the 2-skeleton} if it meets the boundary of every tetrahedron in a collection of normal loops. The {\it complexity} of such a surface is the maximum complexity of all such normal loops.
\end{dfn}

\begin{ex}\
	\begin{enumerate}
		\item Haken \cite{Haken} showed that surfaces that are both incompressible and $\bdy$-incompressible can be isotoped to be normal; i.e., to be complexity 0 surfaces, and in addition to meet each tetrahedron in disks.
		\item It follows from work of Rubinstein \cite{Rubinstein93} and Stocking \cite{Stocking96} that Heegaard surfaces may be isotoped to be complexity at most 1, provided they are {\it strongly irreducible}. This condition means that any compressing disk on one side of such a Heegaard surface meets every compressing disk on the other, and that there is at least one such disk on each side.
		\item It follows from \cite{BDTS2012} that strongly irreducible surfaces that are also {\it $\bdy$-strongly irreducible} can be isotoped to be complexity 1. Such surfaces have the additional property that any compressing or $\bdy$-compressing disk on one side meets every compressing and $\bdy$-compressing disk on the other. 
	\end{enumerate}
\end{ex}

\begin{dfn}
Two surfaces $F$ and $G$ that are normal with respect to the 2-skeleton are {\em compatible} if on the boundary of each tetrahedron $\Delta$, the normal loops $F \cap \bdy \Delta$ and $G \cap \bdy \Delta$ can be normally isotoped to be disjoint.
\end{dfn}

It is well known that if $F$ and $G$ are compatible, and $\alpha$ and $\beta$ are components of $F \cap \bdy \Delta$ and $G \cap \bdy \Delta$, respectively, of length larger than 3, then $\alpha$ and $\beta$ are normally isotopic.

\begin{dfn}
\label{d:compatibilityclass}
A maximal (with respect to inclusion) set of surfaces that are normal with respect to the 2-skeleton and pairwise compatible forms a {\it compatibility class}.
\end{dfn}

A choice of a normal loop of length larger than 3 on the boundary of each tetrahedron thus determines at most one compatibility class.

\begin{dfn}
The {\it complexity} of a non-empty compatibility class is the maximum of the complexities of the surfaces contained in the class.
\end{dfn}

\begin{figure}
\psfrag{1}{$x_1$}
\psfrag{2}{$x_2$}
\psfrag{3}{$x_3$}
\psfrag{4}{$x_4$}
\psfrag{5}{$x_5$}
\psfrag{6}{$x_6$}
\[\includegraphics[width=1.5 in]{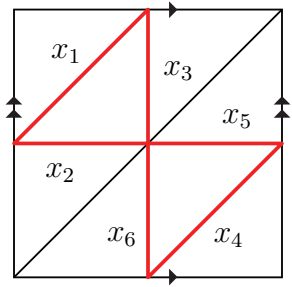}\]
\caption{The six normal arc types in a 1-vertex triangulation of a torus.}
\label{figNormalCurvesInATorus}
\end{figure}

Following \cite{Jaco2003}, a (not necessarily connected) normal curve in a one-vertex triangulation of a torus is parameterized by a free choice of normal coordinates $\{ (x_1,x_2,x_3)\ | \ x_i \in \mathbb N \}$ in exactly one of the two triangles, forcing the normal coordinates in the other triangle to be equal, $x_1=x_4, x_2 = x_5, x_3 = x_6$ (see Figure \ref{figNormalCurvesInATorus}).  Projectively, this set corresponds to the points with rational coordinates in the 2-simplex pictured in Figure \ref{figProjectiveSolutionSpace}.

\begin{figure}
\psfrag{1}{$1$}
\psfrag{2}{$2$}
\psfrag{3}{$3$}
\psfrag{a}{$x_1=1,x_2=0,x_3=0$}
\psfrag{b}{$x_1=0,x_2=1,x_3=0$}
\psfrag{c}{$x_1=0,x_2=0,x_3=1$}
\[\includegraphics[width=2.5 in]{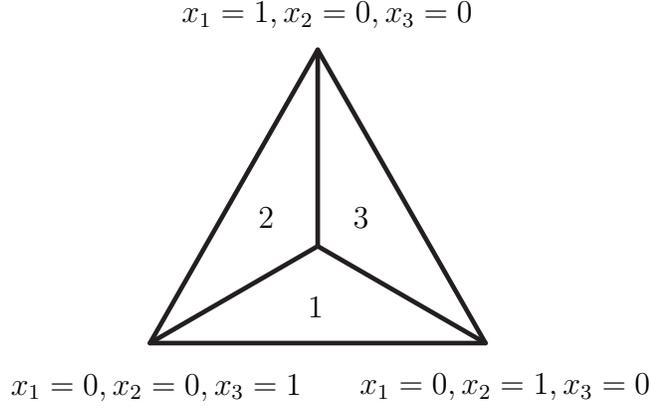}\]
\caption{The projective solution space has three types of normal curves.}
\label{figProjectiveSolutionSpace}
\end{figure}

A normal curve has {\it type} $i$ if $x_i \leq x_j, j \neq i$. Of course, a normal curve with a non-unique minimal coordinate has more than one type, and will be all three types precisely when it consists entirely of trivial curves. %For a triangulated manifold $(M,\mathcal T)$ that restricts to a one-vertex triangulation on each boundary component, define a {\it type} to be a choice of a type for each torus boundary component.

%Assume henceforth that every surface under consideration meets $\partial M$ only in $T_1$ and $T_2$.

\begin{dfn}
\label{d:typed}
The {\em type} of a surface $S$ in $M$ is the pair of types of the boundary curves of $S$ on $T_1$ and $T_2$, respectively. A {\it typed compatibility class} is a compatibility class restricted to surfaces of the same type.
\end{dfn}

\begin{rmk}
\label{r:finite}
Since there are a bounded number of distinct normal isotopy classes of loops of bounded complexity on the boundary of a single tetrahedron, then it follows that there are a bounded number of typed compatibility classes with complexity less than a specified bound. For instance, there are at most $3^{2} 3^t$ typed compatibility classes of complexity zero surfaces, each determined by a pair of boundary types and a choice of a length 4 loop in each of the $t$ tetrahedra.
\end{rmk}
%
%\begin{dfn}
%\label{d:balanced}
%A surface $S$ is {\em balanced} if it has the same number of boundary components on $T_1$ and $T_2$, i.e., $|S \cap T_1| = |S \cap T_2 |$.   A {\em balanced typed compatibility class } is a typed compatibility class restricted to only its balanced surfaces.
%\end{dfn}

%\begin{rmk}
%\label{r:finite}
%Following the discussion preceding Definition~\ref{d:typed}, we see that for any positive integer $c$ there are a finite number of typed compatibility classes of complexity at most $c$.
%\end{rmk}

Let $\alpha$ and $\beta$ be normal arcs in an oriented triangle $\delta$. Then $\alpha$ and $\beta$ can be isotoped, keeping their boundaries fixed, so that they intersect in at most one point.  We define the {\it normal sign} of the point $\alpha \cap \beta$, if it exists, as follows. There is at least one arc $\epsilon$ of the 1-skeleton of $\delta$ that  connects $\alpha$ to $\beta$. Then we say $\alpha \cap \beta$ is

\begin{itemize}
\item {\em positive} if the orientation on $\epsilon$ induced by the orientation of $\delta$ points from $\beta$ to $\alpha$, and
\item {\em negative} otherwise  (see Figure~\ref{figNormalSwitches}).
\end{itemize}

\begin{figure}
\psfrag{a}{$\alpha$}
\psfrag{b}{$\beta$}
\psfrag{A}{$\epsilon$}
\psfrag{+}{$+1$}
\psfrag{-}{$-1$}
\[\includegraphics[width=3 in]{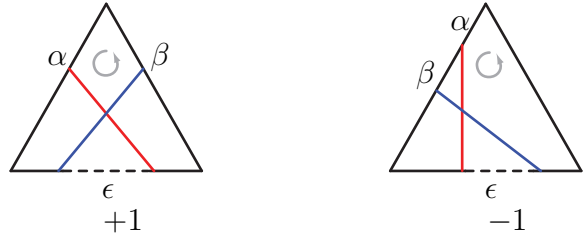}\]
\caption{Positive and negative intersections between normal arcs $\alpha$ and $\beta$.}
\label{figNormalSwitches}
\end{figure}

Lemma 3.5 of \cite{Jaco2003} shows that normal isotopy and isotopy are equivalent for normal curves in a one-vertex triangulation of a torus. Moreover,

\begin{lem}
\label{l:SameSign}
Suppose that normal curves of the same type, $\alpha$ and $\beta$, have been normally isotoped to intersect minimally in a one-vertex triangulation of a torus. Then all intersection points between $\alpha$ and $\beta$ have the same normal sign.
\end{lem}

\begin{proof}
Without loss of generality assume that $\alpha$ and $\beta$ are both of type 3. Since they intersect minimally, each trivial component of $\alpha$ and $\beta$ is disjoint from every other component.  Let $\alpha'$ and $\beta'$ be obtained by removing trivial components from $\alpha$ and $\beta$ respectively. Then  $\alpha'$ and $\beta'$ do not possess normal arcs of the third type and can be consistently oriented by applying orientations to the remaining two arc types as in Figure \ref{figOrientingArcs}.

\begin{figure}
\[\includegraphics[width=1.5 in]{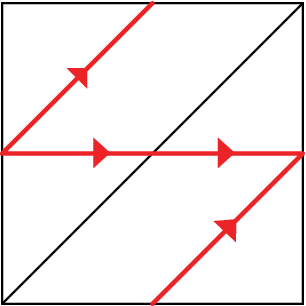}\]
\caption{Type 3 curves can be oriented by orienting normal arcs as indicated.}
\label{figOrientingArcs}
\end{figure}

Now consider an intersection between a normal arc of $\alpha'$ and a normal arc of $\beta'$. Since the curves are of type 3, the algebraic sign is equal to the normal sign of the intersection point (see Figure \ref{figNormalIsAlgebraic}).
%If picture is not convincing this will need more verbage
 We have already observed, however, that  $\alpha'$ and $\beta'$ intersect minimally, hence all intersections have the same algebraic sign.  It follows that all intersection points have positive normal sign or all have negative normal sign.

\begin{figure}
\[\includegraphics[width=1.5 in]{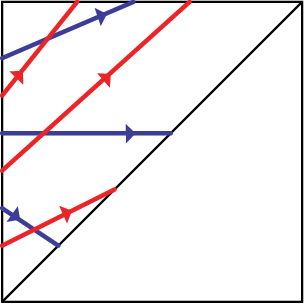}\]
\caption{For intersections between oriented type 3 curves, the algebraic sign is equal to the normal sign. }
\label{figNormalIsAlgebraic}
\end{figure}

\end{proof}

\begin{lem}
\label{l:sameCount}
Let $A$ and $B$ be surfaces meeting the 2-skeleton of $M$ normally, in the same typed compatibility class, whose intersection with $\bdy M$ is contained in $T_1 \cup T_2$. Suppose $A$ and $B$ are normally isotoped to intersect minimally. Let $\alpha_i = A \cap T_i$ and $\beta_i = B \cap T_i$ designate their normal boundary curves, oriented by type (see Figure~\ref{figOrientingArcs}). Then the algebraic intersection numbers are related by $\#(\alpha_1,\beta_1) = -\#(\alpha_2,\beta_2)$.
\end{lem}

%\marginpar{ this is a bit stronger - the normal intersection number on one component is -1 times that of the other. not sure we want to go this far, would need more notation. }

\begin{proof}
Consider a tetrahedron, $\Delta$. Let $\alpha$ denote a component of $A \cap \bdy \Delta$, and $\beta$ a component of $B \cap \bdy \Delta$. Since $A$ and $B$ intersect minimally, we may assume each normal arc of $\alpha$ and $\beta$ is a straight line segment.

As $A$ and $B$ are compatible, there is an isotopy $\alpha_t$ from $\alpha$ to a normal loop $\alpha'$  in $\bdy \Delta$ that is disjoint from $\beta$.   We can choose such an isotopy that interpolates the intersections with the 1-skeleton and so that for all $t$, each normal arc of $\alpha_t$ is a straight line segment.  Furthermore, by adjusting the rate of interpolation, we may assume that the isotopy is in general position.   Let $\{t_i\}$ denote the critical values of $\alpha_t \cap \beta$, i.e.~the values of $t$ such that $\alpha_t$ and $\beta$ do not intersect transversely on $\partial \Delta$. It follows that for each $i$, $\alpha_{t_i} \cap \beta$ includes a point of the 1-skeleton. Let $\eta_t$ denote the difference between the number of positive and negative intersection points of $\alpha_t$ and $\beta$, when their intersection is in general position.

\begin{figure}
\psfrag{a}{$\alpha_{t_i-\epsilon}$}
\psfrag{A}{$\alpha_{t_i+\epsilon}$}
\psfrag{b}{$\beta$}
\[\includegraphics[width=4 in]{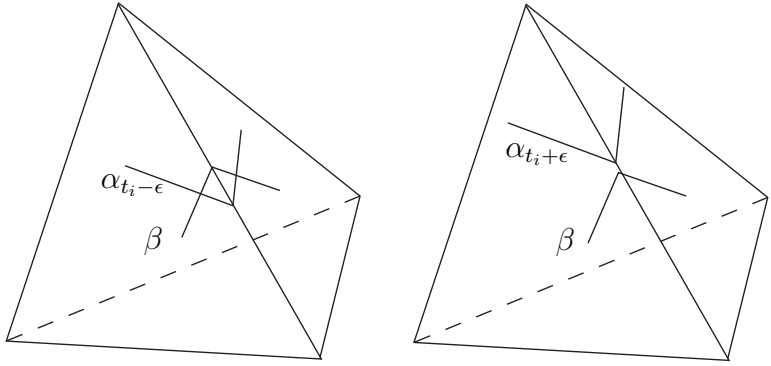}\]
\caption{Positive and negative intersections cancel as $t$ increases through $t_i$.}
\label{f:PositiveNegativeCancellation}
\end{figure}

Just before (or after) $t_i$, $\alpha_t$ meets $\beta$ as in Figure \ref{f:PositiveNegativeCancellation}. Here we see two intersections, one of each normal sign, of $\alpha_t \cap \beta$ which cancel as $t$ increases through $t_i$. It follows that $\eta_{t_i-\epsilon}=\eta_{t_i+\epsilon}$. As $\alpha' \cap \beta=\emptyset$, we conclude $\eta_t$ is zero for all non-critical $t$. It follows that there are an equal number of positive and negative normal signs on the intersections of $\alpha$ and $\beta$ on $\partial \Delta$.

Note that the induced orientations on an interior face $\delta$ are opposite from the two (not necessarily distinct) tetrahedra on either side of it. Therefore the normal sign of each intersection point is opposite with respect to each orientation. It follows that the sum of the normal intersections is zero when summed over only faces in the boundary. By Lemma \ref{l:SameSign}, on each boundary component $T_i$, each intersection point between $\alpha_i$ and $\beta_i$ has the same sign. Therefore the positive and negative intersections between $A$ and $B$ on $T_1 \cup T_2$ happen on different tori. Since they agree with the algebraic intersection numbers there, $\#(\alpha_1,\beta_1) = - \#(\alpha_2,\beta_2)$.
\end{proof}

The above result is particularly relevant to the case that $M$ is connected. In the case that $M$ consists of two components (each having a torus boundary component), Lemma~\ref{l:sameCount} implies that every surface in a given compatibility class (intersecting $\partial M$ in $T_1 \cup T_2$ only) intersects each boundary torus in a unique slope. 

%Maps between the homology groups of a torus can be regarded as elements of $SL_2(\mathbb Z)$, and an element $\Phi \in  PSL_2 \mathbb Q $ acts on the set of slopes $\widehat{ \mathbb Q}$ as a linear fractional transformation, i.e., $\Phi(q) = \frac{aq + b }{cq+d}$.

%Fix an identification of the boundaries, $\iota: T_1 \rightarrow T_2$, that is orientation reversing with respect to induced boundary orientations.

Fix homology bases for $T_1$ and $T_2$ so that the orientation on $T_1$ agrees with the induced boundary orientation, and the orientation on $T_2$ disagrees with the induced boundary orientation.  A {\it slope} is the isotopy class of a closed and connected essential curve in a torus. With our choice of basis for the homology of the torus, slopes are parameterized by $\widehat{ \mathbb Q} = \mathbb Q \cup \{ 1/0  \}$. If $c$ is an embedded but disconnected curve that contains an essential curve, then we will define  the {\it slope of $c$}, denoted $\langle c \rangle$, to be the isotopy class of one of its essential connected components. The {\it slope of a surface $S$} on a torus $T$ is the slope of an essential component of its intersection with $T$. An element $\Phi \in  SL_2 (\mathbb Q) $ acts on the set of slopes $\widehat{ \mathbb Q}$ as a linear fractional transformation, i.e., $\Phi(q) = \frac{aq + b }{cq+d}$.

%A {\it slope} is the isotopy class of a closed and connected essential curve in a torus. With a choice of basis for the homology of the torus, slopes are parameterized by $\widehat{ \mathbb Q} = \mathbb Q \cup \{ 1/0  \}$. If $c$ is an embedded but disconnected curve that contains an essential curve, then we will define  the {\it slope of $c$}, denoted $\langle c \rangle$, to be the isotopy class of one of its essential connected components. Note that if $\gamma = \left[ \begin{array}{c} p \\ q \end{array}\right]$ is an element of homology (such that $(p,q) = 1$), then $\gamma$ and $-\gamma$ correspond to the same slope on the torus. The {\it slope of a surface $S$} on a torus $T$ is the slope of an essential component of its intersection with $T$.

%Fix an identification of the boundaries, $\iota: T_1 \rightarrow T_2$, that is orientation reversing with respect to induced boundary orientations.

%Fix homology bases for $T_1$ and $T_2$ so that the orientation on $T_1$ agrees with the induced boundary orientation and the orientation on $T_2$ disagrees with the induced boundary orientation.

\begin{lem}
\label{SL2Qlemma}
Let $\mathcal C$ be a typed compatibility class of surfaces that meet the 2-skeleton normally. Then there exists an element $\Phi_{\mathcal C} \in SL_2 (\mathbb Q) $ such that for all $S \in \mathcal C$ where $S \cap T_1 \neq \emptyset$ , $S \cap T_2 \neq \emptyset$, and   $S \cap \bdy M \subset T_1 \cup T_2$, the element $\Phi_{\mathcal C}$ maps the slope of $S$ on $T_2$ to the slope of $S$ on $T_1$, i.e. $\Phi_{\mathcal C} (\langle S \cap T_2 \rangle) = \langle  S \cap T_1 \rangle $.
\end{lem}
\begin{proof}

There are two cases. In the first case, we assume that all surfaces $S \in  \mathcal C$ meet $T_2$ in a single slope $\tau_2$.  Since any pair of these surfaces can be isotoped to be disjoint on $T_2$, it follows from Lemma \ref{l:sameCount} that they can also be isotoped to be disjoint on $T_1$. Hence all surfaces $S \in \mathcal C$ have the same slope $\tau_1$ on $T_1$.   With respect to the given bases on $T_i$ choose a particular element $\Phi_{\mathcal C} \in SL_2 (\mathbb Q)$ so that $\Phi_{\mathcal C} ( \tau_2 ) = \tau_1 $.

In the second case, we suppose that there are two surfaces $R$ and $S$ in $\mathcal C$ which meet $T_2$ in distinct slopes, and hence in distinct slopes on $T_1$.  Let $\vec r_i$ and $\vec s_i$ be column vectors expressing the homology classes  $[R \cap T_i]$ and $[S \cap T_i]$, respectively.  Construct the matrices $\Psi_i = ( \vec r_i \vec s_i )$.   By Lemma \ref{l:sameCount} and because one of the chosen homology bases disagrees with the induced boundary orientations, $det( \Psi_1) = det (\Psi_2)$.   Now let $\Phi_{\mathcal C} = \Psi_1 \Psi_2^{-1}$.  Note that $\Phi_{\mathcal C}$ has determinant 1, and thus  can be regarded as an element of $SL_2(\mathbb Q)$. Restricting the equation $\Phi_{\mathcal C} \Psi_2 = \Psi_1$ to each column vector of $\Psi_2$, we obtain $\Phi_{\mathcal C} \vec r_2 =  \vec r_1$ and $\Phi_{\mathcal C} \vec s_2 = \vec s_1 $, i.e., $\Phi_{\mathcal C}$ is a map between the homology classes of the boundaries of these surfaces.

We need to show this holds for any other surface $Q \in \mathcal C$.   Suppose that $Q$ meets the boundary with homology classes $\vec q_1$ and $\vec q_2$. Since $\vec r_1$ and $\vec s_1$ are independent, any other vector is uniquely determined by its intersections with them.   We check that $\vec q_1$ and $\Phi_\mathcal C  \vec q_2$ have the same number of intersections with $\vec{r_1}$:

$$det( \Phi_\mathcal C  \vec q_2, \vec r_1 ) = det( \Phi_\mathcal C  \vec q_2, \Phi_\mathcal C  \vec r_2 ) = det( \Phi_\mathcal C ) \, det(\vec q_2,\vec r_2) = det(\vec q_1, \vec r_1),$$ by Lemma \ref{l:sameCount}.  The calculation for  $\vec s_1$ is identical and thus $\Phi_\mathcal C $ maps $\vec q_2$ to $\vec q_1$, as desired.

Finally, in this second case, we regard $\Phi_\mathcal C $ as a linear fractional transformation so that it acts on slopes rather than homology classes.
\end{proof}

We now define the notion of distance of a map that glues together two torus boundary components of a 3-manifold. To do this, we first define the Farey graph:

%%%%
\begin{dfn}
Let $T$ be a torus. The {\it Farey graph} for $T$ is the graph formed by taking a vertex for each slope on $T$, and connecting two vertices with an edge if simple closed curve representatives of the two corresponding slopes can be isotoped to intersect in one point on $T$. The distance between two slopes is taken to be the minimum length of a path in the Farey graph between the corresponding vertices. 
\end{dfn} 

\begin{dfn}
\label{d:c-distance}
Fix a number $c$. For each typed compatibility class $\mathcal C$ of surfaces of complexity at most $c$ (of which there are a finite number by Remark~\ref{r:finite}), let $\Phi_{\mathcal C}$ be the element of $SL_2 (\mathbb Q)$ given by Lemma~\ref{SL2Qlemma}. Suppose $\phi \colon T_1 \to T_2$ is a gluing map. Then we say $\phi \Phi_{\mathcal C}$ has {\it distance} $d$ if the minimum distance (in the Farey graph of $T_2$), over all slopes $\gamma$ on $T_2$, of $\gamma$ and the slope of $\phi \Phi_{\mathcal C} (\gamma)$, is $d$. The {\it $c$-distance} of the gluing map $\phi$ is then the minimum, over all maps $\Phi_{\mathcal C}$, of the distances of $\phi \Phi_{\mathcal C}$. 
\end{dfn}

%Fix a number $c$. For each typed compatibility class $\mathcal C$ of surfaces of complexity at most $c$ (of which there are a finite number by Remark~\ref{r:finite}), let $\Phi_{\mathcal C}$ be the element of $SL_2 (\mathbb Q)$ given by Lemma~\ref{SL2Qlemma}. We will abuse notation slightly and take $\Phi_{\mathcal C}$ to be both a map on homology and a map between the tori $T_2$ and $T_1$. \marginpar{I added this sentence because I was unhappy with talking about $\Phi_{\mathcal C}$ mapping curves, when it originally was construted as a map between homology classes. But $\Phi_{\mathcal C}$ isn't necessarily a homeomorphism, so how do we talk about this map?} Suppose $\phi \colon T_1 \to T_2$ is a gluing map. We say a curve $\gamma$ is {\it balanced} with respect to $\Phi_{\mathcal C}$ if $\phi \Phi_{\mathcal C} (\gamma)$ has the same number of components as $\gamma$. Then we say $\phi \Phi_{\mathcal C}$ has {\it distance} $d$ if the minimum distance (in the Farey graph of $T_2$), over all curves $\gamma$ on $T_2$ that are balanced with respect to $\Phi_{\mathcal C}$, between the slope of $\gamma$ and the slope of $\phi \Phi_{\mathcal C} (\gamma)$, is $d$. The {\it $c$-distance} of the gluing map $\phi$ is then the minimum, over all maps $\Phi_{\mathcal C}$, of the distances of $\phi \Phi_{\mathcal C}$. If no curve $\gamma$ is balanced with respect to any map $\Phi_{\mathcal C}$, then we say $\phi$ has infinite $c$-distance.
%\end{dfn}

Note that if the $c$-distance of $\phi$ is non-zero then each composition $\phi \Phi_{\mathcal C}$ fixes no slope. Also, observe that if $c \leq c'$, then the $c$-distance of $\phi$ is greater than or equal to the $c'$-distance of $\phi$.

\section{Proof of Theorem \ref{main_thm1}.}
\label{s:IntersectingWithTori}

In this section we present the proof of Theorem \ref{main_thm1}. This is broken up into two cases, depending on whether the surfaces we are concerned with are incompressible or strongly irreducible.

\subsection{The incompressible case.}\

\begin{lem}
\label{l:IncompressibleTorusIntersection}
Let $X$ be a compact, irreducible, orientable 3-manifold with $\partial X$ incompressible, if non-empty. Suppose $S$ and $T$ are closed, essential surfaces in $X$. Then $S$ may be isotoped to be transverse to $T$ with every component of $S \setminus N(T)$ incompressible in the respective submanifold of $X \setminus N(T)$.
\end{lem}

\begin{proof}
Isotope $S$ so that it meets $T$ transversally, and so that $|S \cap T|$ is minimal. Let $\alpha$ be a loop of $S \cap T$. Let $D$ be a compressing disk for $S \setminus N(T)$. As $S$ is incompressible in $X$, $\bdy D$ bounds a subdisk $E$ of $S$, where necessarily $E \cap T \ne \emptyset$. As $T$ is incompressible, every loop of $E \cap T$ must be inessential on $T$. We conclude $S \cap T$ contains some loop that is inessential on both surfaces. But then by a standard innermost disk argument we can lower $|S \cap T|$, a contradiction. Thus, $S \setminus N(T)$ must be incompressible in $X \setminus N(T)$.

We must now establish that each component of $S \setminus N(T)$ is incompressible. Choose the compressing disk $C$ for a component of $S \setminus N(T)$ whose interior meets the other components in a minimal number of loops. Let $\alpha$ be such a loop that is innermost on $C$. Then $\alpha$ bounds a subdisk $C'$ of $C$, and $C' \cap S=\bdy C'$. As $S \setminus N(T)$ is incompressible, $\bdy C'$ is inessential on $S \setminus N(T)$. But then by a similar innermost disk argument we may reduce the number of intersections of $C$ with $S \setminus N(T)$, a contradiction.
\end{proof}

\begin{lem}
\label{l:IncompImpliesBdyIncomp}
Let $M$ be a compact, orientable, irreducible 3-manifold. Let $S$ be an orientable, connected, properly embedded, incompressible surface in $M$, such that each loop of $\bdy S$ is contained in a torus component of $\bdy M$. Then either $S$ is a boundary-parallel annulus or $S$ is $\bdy$-incompressible.
\end{lem}

\begin{proof}
Suppose there is a $\bdy$-compression $D'$ for $S$. Let $T$ be the component of $\bdy M$ that meets $D'$. Then $D' \cap T$ is an arc, $\beta$. There are now two cases. Suppose first $\bdy \beta$ connects two distinct loops of $S \cap T$. As $T$ is a torus, these loops cobound an annulus $A$ of $T$. Let $A'$ denote the disk obtained from $A$ by removing a neighborhood of $\beta$. Then two parallel copies of $D'$, together with the disk $A'$, is a disk whose boundary is on $S$. As $S$ is incompressible, we conclude $S$ is a boundary-parallel annulus.

The second case is when $\bdy \beta$ connects some loop of $S \cap T$ to itself. Then $\beta$, together with an arc of $S \cap T$, cobounds a subdisk $A''$ of $T$. The disk $D' \cup A''$ is then a compressing disk for $S$, a contradiction.
\end{proof}

We now establish the incompressible case of Theorem \ref{main_thm1}. Recall that $M$ is assumed to be a 3-manifold with incompressible boundary, $\phi \colon T_1 \to T_2$ is a gluing map between two torus components of $\bdy M$, and $M_\phi$ is the resulting 3-manifold. Let  $T$ be the image of $T_1$ in $M_\phi$. Fix a triangulation $\mathcal T$ of $M$ as in Section \ref{s:Normal-like}. Let $S$ be a closed, orientable,  incompressible surface in $M_\phi$.

By Lemma \ref{l:IncompressibleTorusIntersection} we can isotope $S$ so that $S'=S \setminus N(T)$ is incompressible in $M_\phi \setminus N(T) \cong M$. If, for some $i$, every component of $S'$ that meets $T_i$ can be isotoped into a neighborhood of $T_i$, then we can isotope $S$ in $M_\phi$ to miss $T$ entirely, and the result follows. Otherwise, by Lemma \ref{l:IncompImpliesBdyIncomp}, for each $i$ there is a component of $S'$ that meets $T_i$ and is both incompressible and $\bdy$-incompressible. Let $S''$ then be a (possibly disconnected) subsurface of $S'$, which is properly embedded in $M$, is incompressible, $\bdy$-incompressible, and meets both $T_1$ and $T_2$.

By \cite{Haken} $S''$ can be isotoped to be a normal surface with respect to $\mathcal T$. In this position, $S''$ is a complexity zero surface in $M$ such that $\langle \phi( S'' \cap T_1 ) \rangle= \langle S'' \cap T_2 \rangle$.

Let $\mathcal C$ denote the typed compatibility class of $S''$. Then by Lemma \ref{SL2Qlemma}, $\Phi_{\mathcal C}(\langle  S'' \cap T_2 \rangle) = \langle S'' \cap T_1\rangle$. It follows that $\phi \Phi_{\mathcal C}$ fixes the slope $S'' \cap T_1$, and thus the map $\phi$ has $0$-distance (and hence 1-distance) zero. This concludes the proof of the incompressible case of Theorem \ref{main_thm1}.

\subsection{The strongly irreducible case.}

When the surface $S$ is strongly irreducible, we may appeal to the results of \cite{Bachman2006}. Lemmas \ref{l:StronglyIrreducibleSweepout} and \ref{l:StronglyIrreducibleImpliesDistanceOne} below, taken from that paper,  will play the roles that Lemmas \ref{l:IncompressibleTorusIntersection} and  \ref{l:IncompImpliesBdyIncomp} played in the previous subsection.

\begin{lem}
\label{l:StronglyIrreducibleSweepout}
\cite{Bachman2006} Let $X$ be a compact, irreducible, orientable 3-manifold with $\partial X$ incompressible, if non-empty. Suppose $X = V \cup _S W$, where $S$ is a strongly irreducible Heegaard surface. Suppose further that $X$ contains an incompressible, orientable, closed, non-boundary parallel surface $T$. Then either
\begin{itemize}
	\item $S$ may be isotoped to be transverse to $T$, with every component of $S \setminus N(T)$ incompressible in the respective submanifold of $X \setminus N(T)$, or
	\item $S$ may be isotoped to be transverse to $T$, with every component of $S \setminus N(T)$ incompressible in the respective submanifold of $X \setminus N(T)$ except for exactly one strongly irreducible component, or
	\item $S$ may be isotoped to meet $T$ in a single saddle tangency (and transverse elsewhere), with every component of $S \setminus N(T)$ incompressible in the respective submanifold of $X \setminus N(T)$.
\end{itemize}
\end{lem}

When $T$ is a torus we can rule out the third conclusion of this lemma, and thereby obtain the following corollary:

\begin{cor}
\label{c:StronglyIrreducibleSweepout}
Let $X$ be a compact, irreducible, orientable 3-manifold with $\partial X$ incompressible, if non-empty. Suppose $X=V \cup _S W$, where $S$ is a strongly irreducible Heegaard surface. Suppose further that $X$ contains an essential torus $T$. Then $S$ may be isotoped to be transverse to $T$ with every component of $S \setminus N(T)$ incompressible in the respective submanifold of $X \setminus N(T)$ except for at most one strongly irreducible component.
\end{cor}

\begin{proof}
Suppose $S$ has been isotoped to meet $T$ in a single saddle tangency (and transverse elsewhere), with every component of $S \setminus N(T)$ incompressible in the respective submanifold of $X \setminus N(T)$. Then $\bdy (S \setminus N(T))$ must be essential on $S$ and both copies of $T$ on $\bdy N(T)$. This implies that isotoping $S$ by pushing the saddle tangency slightly past $T$ in either direction yields essential curves of intersection on $T$. This is impossible, however, as $T$ is a torus, and resolving a saddle tangency in this way must produce an inessential curve on one side or the other. 
\end{proof}

\begin{lem}
\label{l:StronglyIrreducibleImpliesDistanceOne}
\cite{Bachman2006} Let $M$ be a knot manifold. Let $S$ be a separating, properly embedded, connected surface in $M$ which is strongly irreducible, has non-empty boundary, and is not peripheral. Then either $S$ is $\partial$-strongly irreducible or $\partial S$ is at most distance one from the boundary of some properly embedded surface which is both incompressible and $\partial$-incompressible.
\end{lem}

In \cite{Bachman2006} the term {\it knot manifold} was used to refer to a 3-manifold with a single boundary component, which is homeomorphic to a torus. However, the proof of Lemma \ref{l:StronglyIrreducibleImpliesDistanceOne} given in \cite{Bachman2006} only uses the fact that each boundary component of the surface $S$ is contained in a torus component of $\bdy M$.

We now proceed with the proof of the strongly irreducible case of Theorem \ref{main_thm1}. As before, $M$ is a 3-manifold with incompressible boundary, $\phi \colon T_1 \to T_2$ is a gluing map between two torus components of $\bdy M$, $M_\phi$ is the resulting 3-manifold, and  $T$ is the image of $T_1$ in $M_\phi$. Fix a triangulation $\mathcal T$ of $M$ as in Section \ref{s:Normal-like}. Let $S$ be a closed, strongly irreducible surface in $M_\phi$.

By Corollary \ref{c:StronglyIrreducibleSweepout} we can isotope $S$ so that every component of $S'=S \setminus N(T)$ is incompressible in $M_\phi \setminus N(T) \cong M$, with the exception of at most one strongly irreducible component. If, for some $i$, every component of $S'$ that meets $T_i$ can be isotoped into a neighborhood of $T_i$, then we can isotope $S$ in $M_\phi$ to miss $T$ entirely, and the result follows. Otherwise, either by Lemma \ref{l:IncompImpliesBdyIncomp} or by Lemma \ref{l:StronglyIrreducibleImpliesDistanceOne} there is a (possibly disconnected) surface $S''$, which is properly embedded in $M$, is either incompressible and $\bdy$-incompressible or strongly irreducible and $\bdy$-strongly irreducible, and meets both $T_1$ and $T_2$. Furthermore, the slope of $\phi(S'' \cap T_1)$ is at most distance one from the slope of $S'' \cap T_2$ in the Farey graph of $T_2$.

By \cite{Haken} and \cite{BDTS2012} the surface $S''$ can be isotoped to be a complexity zero or one surface in $M$. Let $\mathcal C$ denote the typed compatibility class of $S''$. Then by Lemma \ref{SL2Qlemma}, $\Phi_{\mathcal C} (\langle S'' \cap T_2 \rangle) = \langle S'' \cap T_1 \rangle$. It follows that $\phi \Phi_{\mathcal C}$ translates the slope of $S'' \cap T_2$ at most distance one in the Farey graph of $T_2$, and thus the map $\phi$ has 1-distance zero or one. This concludes the proof of Theorem \ref{main_thm1}.

\section{Heegaard splittings of 3-manifolds with sufficiently complicated JSJ decompositions are amalgamations.}
\label{s:selfamalgamation}

Generalized Heegaard splittings were first introduced by Scharlemann and Thompson in \cite{st:94}. Here we review some of the basic definitions of this theory (articulated with the language and notation given in \cite{gordon}) and provide an application of Theorem \ref{main_thm1}.

\begin{dfn}
A {\it compression body} is a 3-manifold which can be obtained by starting with some closed, orientable, connected surface, $H$, forming the product $H \times I$, attaching some number of 2-handles to $H \times \{0\}$, and capping off all resulting 2-sphere boundary components that are not contained in $H \times \{1\}$ with 3-balls. The boundary component $H \times \{1\}$ is referred to as $\partial _+$. The rest of the boundary is referred to as $\partial _-$.
\end{dfn}

\begin{dfn}
A {\it Heegaard splitting} of a 3-manifold $M$ is an expression of $M$ as a union $V \cup _H W$, where $V$ and $W$ are compression bodies that intersect in a transversally oriented surface $H=\partial _+ V=\partial _+ W$. If $V \cup _H W$ is a Heegaard splitting of $M$ then we say $H$ is a {\it Heegaard surface}. 
\end{dfn}

\begin{dfn}
A complex $\Sigma$ is the {\it spine} of a Heegaard splitting $V \cup _H W$ of $M$ if
	\begin{enumerate}
		\item $\Sigma$ is the union of some subset of $\bdy M$ and a properly embedded graph in $M$.
		\item $H$ is parallel to the frontier of a neighborhood of $\Sigma$ in $M$.
	\end{enumerate}
\end{dfn}

\begin{dfn}
\label{d:GHS}
A {\it generalized Heegaard splitting (GHS)} $\mathcal H$ of a 3-manifold $M$ is a pair of sets of pairwise disjoint, transversally oriented, connected surfaces,  $\thick{\mathcal H}$ and $\thin{\mathcal H}$ (called the {\it thick levels} and {\it thin levels}, respectively), which satisfies the following conditions.
	\begin{enumerate}
		\item Each component $M'$ of $M-\thin{\mathcal H}$ meets a unique element $H_+$ of $\thick{\mathcal H}$, and $H_+$ is a Heegaard surface in $M'$. Henceforth we will denote the closure of the component of $M-\thin{\mathcal H}$ that contains an element $H_+ \in \thick{\mathcal H}$ as $M(H_+)$.
		\item As each Heegaard surface $H_+ \subset M(H_+)$ is transversally oriented, we can consistently talk about the points of $M(H_+)$ that are ``above'' $H_+$ or ``below'' $H_+$. Suppose $H_-\in \thin{\mathcal H}$. Let $M(H_+)$ and $M(H_+')$ be the submanifolds on each side of $H_-$. Then $H_-$ is below $H_+$ if and only if it is above $H_+'$.
		\item There is a partial ordering on the elements of $\thin{\mathcal H}$ which satisfies the following: Suppose $H_+$ is an element of $\thick{\mathcal H}$, $H_-$ is a component of $\partial M(H_+)$ above $H_+$, and $H'_-$ is a component of $\partial M(H_+)$ below $H_+$. Then $H_- > H'_-$.
	\end{enumerate}
\end{dfn}

\begin{dfn}
A GHS $\mathcal H$ of $M$ is {\it strongly irreducible} if every element $H_+\in \thick{\mathcal H}$ is strongly irreducible in $M(H_+)$.
\end{dfn}

\begin{thm}
\label{t:ThinLevelsIncompressible}
\cite{st:94} The thin levels of a strongly irreducible GHS are incompressible.
\end{thm}

\begin{cor}
\label{c:descendendant}
Let $M$ be a 3-manifold with incompressible boundary and let $T_1$ and $T_2$  be two torus components of $\partial M$. Let $\phi \colon T_1 \to T_2$ be a gluing map whose 1-distance is at least two, and let $T$ be the image of $T_1$ in $M_\phi$. Then $T$ is isotopic to a thin level of every strongly irreducible GHS of $M_\phi$.
\end{cor}

\begin{proof}
Let $\mathcal H$ be a strongly irreducible GHS of $M_\phi$. By Theorem \ref{t:ThinLevelsIncompressible}, the thin levels of $\mathcal H$ are incompressible. By Theorem \ref{main_thm1}, such surfaces can be isotoped to be disjoint from $T$. But then $T$ lies in $M(H_+)$, for some thick level $H_+$. As $H_+$ is strongly irreducible we may again apply Theorem \ref{main_thm1} to conclude $T$ can be isotoped be disjoint from $H_+$ in $M(H_+)$. But then $T$ lies in a compression body. The only incompressible surfaces
in a compression body are isotopic to negative boundary components, and all such components are thin levels of $\mathcal H$.
\end{proof}

%In \cite{gordon} and \cite{Lackenby2008} it was shown how 

The GHS $\mathcal H$ determines a unique Heegaard splitting of $M$, called its {\it amalgamation}, a generalization of a procedure with the same name given in \cite{Schultens93}.  The definition provided here follows the version given in \cite{gordon}, and is illustrated schematically in Figure \ref{f:amalgamation}. The dark lines in the figure represent a complex which is the spine of a Heegaard splitting of the manifold $M$ pictured. This complex has three parts. The first are those subsurfaces of $\bdy M$ that lie below thick levels. The second part is a collection of loops that lies entirely in compression bodies below thick levels. Finally, the third part is a collection of vertical arcs that connect the other two parts through compression bodies that lie above thick levels.

\begin{figure}
\[\includegraphics[width=4 in]{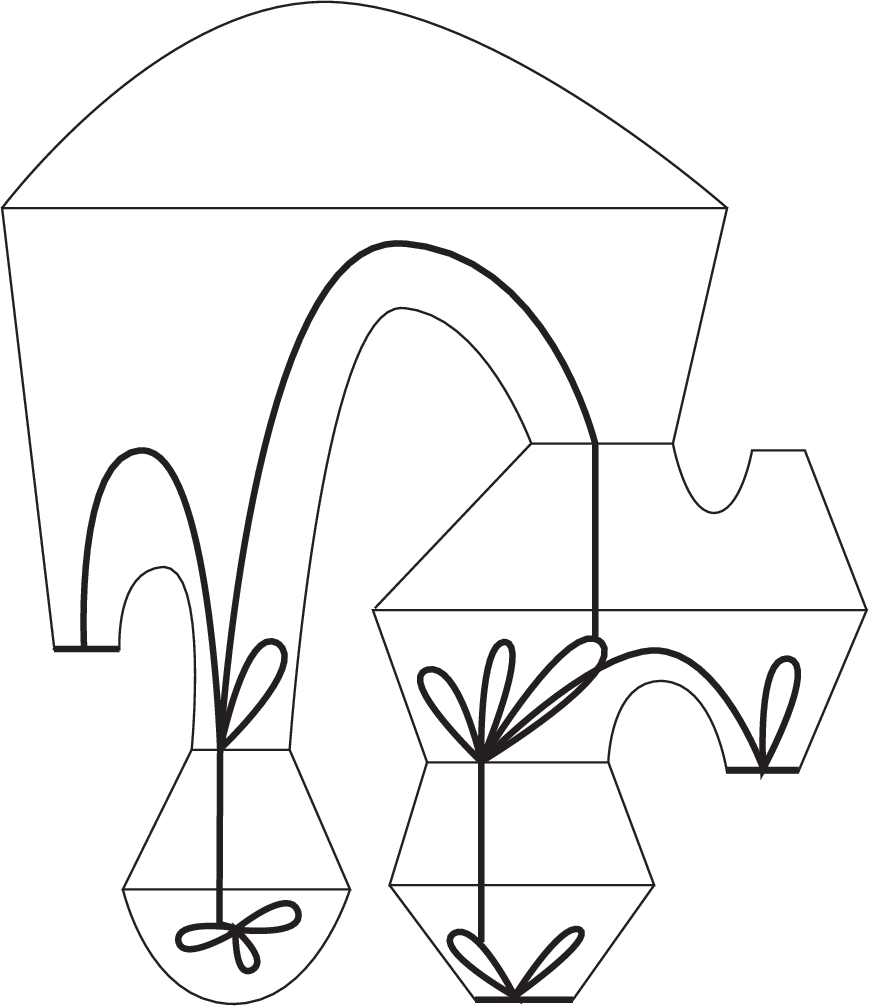}\]
\caption{The wider horizontal lines are the thick levels of a GHS, the narrower horizontal lines represent its thin levels, and the darker arcs represent the spine of its amalgamation.}
\label{f:amalgamation}
\end{figure}

The following result about amalgamation is implicit in \cite{st:94}:

\begin{lem}
\label{l:StronglyIrreducibleAmalgamation}
Suppose $M$ is a compact, orientable 3-manifold with incompressible boundary. Then every Heegaard splitting of $M$ is the amalgamation of some strongly irreducible GHS.
\end{lem}

\begin{dfn}
Suppose $\{M_i\}$ is a collection of 3-manifolds, and $M$ is obtained from this collection by gluing pairs of boundary components by a collection of maps $\{\phi_j\}$. Let $\mathcal H$ be a GHS of $M$ whose thin levels are precisely the images of the boundary components of $\{M_i\}$ in $M$. Then the thick levels of $\mathcal H$ are a set $\{H_i\}$, where $H_i$ is a Heegaard surface in $M_i$. The amalgamation $H$ of $\mathcal H$ is then said to be the {\it amalgamation of the Heegaard surfaces $H_i$} over the maps $\{\phi_j\}$.
\end{dfn}

Now we extend the notion of $c$-distance to gluings along multiple tori.

\begin{dfn}
\label{d:Generalizedc-distance}
Let $\mathbf T$ denote a collection of pairwise disjoint, pairwise non-parallel, essential tori in a compact, orientable 3-manifold $M$. Let $\{T_i\}_{i=1}^n$ be an ordering of the elements of $\mathbf T$. Let $M_0$ denote the manifold obtained from $M$ by cutting along all tori in $\mathbf T$. Let $T_i'$ and $T_i''$ denote the copies of $T_i$ on $\bdy M_0$. Let $M_i$ denote the manifold obtained from $M_{i-1}$ by the gluing $\phi_i \colon T_i' \to T_i''$, defined so that the manifold $M_n \cong M$. Then we say the ordered collection $\{T_i\}$ has $c$-distance $d$ if there is a triangulation of each manifold $M_i$ so that the minimum of the $c$-distances of the maps $\phi_i$ is precisely $d$. The $c$-distance of $\mathbf T$ is the maximum $c$-distance taken over all of its orderings $\{T_i\}$.
%The set $\mathbf T$ has $c$-distance $d$ if it has an ordering $\{T_i\}$ with $c$-distance $d$. 
\end{dfn}
 
The following is a celebrated theorem of 3-manifold topology. 

\begin{thm} \cite{JacoShalen}, \cite{Johannson2}
\label{d:JSJdecomposition}
Let $M$ be a closed, irreducible, orientable 3-manifold. Then there exists a minimal, pairwise disjoint union $\mathbf T$ of tori embedded in $M$ such that each component of $M - \mathbf T$ is Seifert fibered or atoroidal. 
\end{thm}

\begin{dfn}
The decomposition of $M$ given in Theorem~\ref{d:JSJdecomposition} is called the {\em JSJ decomposition} of $M$, and $\mathbf T$ is called the {\em collection of JSJ tori} in $M$. 
\end{dfn}

%Our main result says that all Heegaard splittings of a 3-manifold admitting a JSJ decomposition with $1$-distance at least two are amalgamations over the gluing maps.

%\begin{thm}
%Let $M$ be a compact, irreducible, orientable 3-manifold that admits a JSJ decomposition whose JSJ tori have $1$-distance at least 2. Then every Heegaard splitting of $M$ is an amalgamation of Heegaard splittings of the characteristic submanifolds of $M$ over the gluing maps.
%\end{thm}

We are now ready to prove Theorem~\ref{JSJtheorem}. 

\begin{proof}
Let $H$ be a Heegaard surface of $M$, and let $\mathbf T$ denote the collection of JSJ tori in $M$. Let $\{T_i\}_{i=1}^n$ be an ordering of the JSJ tori realizing the 1-distance of $\mathbf T$, and let $\{M_i\}$ be the corresponding submanifolds as given in Definition \ref{d:Generalizedc-distance}. Now for some $i$, $H$ is an amalgamation of splittings of the components of the manifold $M_i$, where at worst $i=n$. Let $m$ be the smallest value of $i$ such that $H$ arises from an amalgamation of splittings of the components of $M_m$ over the maps $\{\phi_i\}_{i=m+1}^n$. We claim $m=0$, and thus the result follows.

Let $H_m$ denote the (possibly disconnected) Heegaard surface of $M_m$ whose amalgamation over the maps $\{\phi_i\}_{i=m+1}^n$ is $H$. By Lemma \ref{l:StronglyIrreducibleAmalgamation} $H_m$ is the amalgamation of a strongly irreducible GHS $\mathcal H_m$ of $M_m$.

By definition, $M_m$ is obtained from $M_{m-1}$ by gluing two components of $\bdy M_{m-1}$ by a map $\phi_m:T_m' \to T_m''$ whose $1$-distance is at least two. The image of $T_m'$ in $M_m$ is an essential torus $T_m$. By Corollary \ref{c:descendendant}, $T_m$ is a thin level of $\mathcal H_m$. We may thus obtain a strongly irreducible GHS $\mathcal H_{m-1}$ of $M_{m-1}$ from $\mathcal H_m$ by replacing $T_m$ in $\thin{\mathcal H_m}$ by the set $\{T_m',T_m''\}$.

Let $H_{m-1}$ denote the amalgamation of $\mathcal H_{m-1}$. Then by definition, $H_m$ is the amalgamation of the components of $H_{m-1}$ along the map $\phi_m$. It follows that $H$ is the amalgamation of the components of $H_{m-1}$ along the maps $\{\phi_i\}_{i=m}^n$. We have thus contradicted the minimality of our choice of $m$.
\end{proof}

A consequence of Theorem~\ref{JSJtheorem} is that 3-manifolds admitting ``sufficiently complicated'' JSJ decompositions in the above sense have only finitely many Heegaard splittings of a given genus, and thus do not satisfy the converse of the generalized Waldhausen Conjecture. 

\begin{cor}
\label{c:finiteness}
Let $M$ be a closed, irreducible, orientable 3-manifold that admits a JSJ decomposition whose collection of JSJ tori has 1-distance at least 2. Then for every positive integer $g$, $M$ admits at most finitely many Heegaard splittings of genus $g$ up to isotopy.
\end{cor}

\begin{proof}
By Theorem~\ref{JSJtheorem}, every Heegaard splitting of $M$ is an amalgamation of Heegaard surfaces of the components of the JSJ decomposition of $M$. Since the amalgamation of Heegaard surfaces is unique up to isotopy, we only need to verify that each component of the JSJ decomposition admits finitely many Heegaard surfaces of a given genus up to isotopy. This follows from \cite{LustigMoriah1991} and \cite{Schultens95} in the case that a component is Seifert fibered, and from \cite{Lackenby2008} (see page 2) in the case a component is atoroidal (see also Theorem 32.17 in \cite{Johannson}). \end{proof}

\section{Gluing maps with $c$-distance at least one}
\label{s:MainTheoremSection}

In this final section, we prove Theorem~\ref{main_thm2}. Recall that $M$ is a 3-manifold with tori $T_1$ and $T_2$ contained in its boundary. To prove the theorem, we exploit the fact that homeomorphisms between $T_1$ and $T_2$ can be considered as matrices in $SL_2(\mathbb Z)$. Therefore, to show that a gluing map $\sigma^n \psi$ has $c$-distance at least one, it suffices to show that for each $c$, the corresponding element in $SL_2(\mathbb Z)$ composed with each $\Phi_{\mathcal C}$ from Lemma~\ref{SL2Qlemma} fixes no slopes. We do this by first establishing several linear algebraic results.

\begin{dfn}
Let $L \in SL_2(\mathbb Q)$, written in matrix form.  Then the {\it denominator} of $L$, denoted $d(L)$, is the least integer $d \geq 1$ so that all entries  of the matrix $d L$ are integers.
\end{dfn}

\begin{dfn}
Let  $L \in SL_2(\mathbb R)$, and let $\vec v = \left[ \begin{array}{c} v_0 \\ v_1 \end{array} \right]$ be an eigenvector of $L$.  Then $r = v_0/v_1 \in  \widehat  {\mathbb R} = \mathbb R \cup \{ \infty \}$ is an {\it eigenslope} of $L$.  An eigenslope $r$ is {\it rational} if $r \in \widehat {\mathbb Q} = \mathbb Q \cup \{ \infty \}$.
\end{dfn}

\begin{lem}
\label{lemDenominator}
Suppose that  $L \in SL_2(\mathbb Q)$, $r$ and $s$ are non-zero integers, and $\vec u$ and $\vec v $ are vectors with coprime integer entries so that $L r \vec u = s \vec v$.    Then $d(L) \geq |s/r| \geq 1/d(L)$, where $d(L)$ is the denominator of $L$.
\end{lem}

\begin{proof}
Suppose that $L r \vec u = s \vec v$, where $\vec u$ and $\vec v$ have coprime integer entries. We first show that $|s/r| \geq 1/d(L)$. Notice $L r \vec u = r/d(L) ( d(L) L \vec u) = s \vec v$, where $d(L) L \vec u$ must have integer entries. Since $\vec v$ has coprime entries, $r/d(L)$ must be a divisor of $s$. It follows that $|s| \geq |r|/d(L)$, and hence $|s/r| \geq 1/d(L)$.

Now we show that $d(L) \geq |s/r|$. We have that $r \vec u = L^{-1} s \vec v$. Note that $d(L^{-1})=d(L)$. It follows from above that $|r/s| \geq 1/d(L)$, and hence $d(L) \geq |s/r|$ as desired.
\end{proof}

\begin{lem}
\label{l:eigenslope}
Let $L \in SL_2(\mathbb Q)$ so that $|trace(L)| < 2/d(L)$ or $|trace(L)| > 2d(L)$.  Then $L$ has no rational eigenslopes.
\end{lem}

\begin{proof}
We prove the contrapositive. Suppose that $L$ has a rational eigenslope written $r/s$, where $r$ and $s$ are coprime integers.  Then $\vec u = \left[\begin{array}{c}r \\ s\end{array}\right]$ is an eigenvector for $L$.  Hence $L \vec u = q \vec u$ where $q$ is the eigenvalue corresponding to $\vec u$. Since $L \in SL_2(\mathbb Q)$ it follows that $q \in \mathbb Q$.  By Lemma \ref{lemDenominator}, we have $d(L) \geq |q| \geq 1/d(L)$. The eigenvalues of $L$ are $q$ and $1/q$, either both positive or both negative. In particular, $|q + 1/q| = |q| + |1/q|$. Moreover, $d(L) \geq |q| \geq 1/d(L)$ implies that $1/d(L) \leq | 1/q| \leq d(L)$. Hence $2d(L) \geq |q| + |1/q| =  |q + 1/q| = |trace(L)|  \geq 2/d(L).$

\end{proof}

\begin{lem}
\label{fixed_point_lem}
Let $J \in SL_2(\mathbb Z)$ and $K \in SL_2(\mathbb Q)$ so that $trace(J)>2$.  Then there exists an $N \in \mathbb N$ so that for all $n \geq N$, either $|trace(J^nK)| < 2/d(J^nK)$ or $|trace(J^nK)| > 2 d(J^nK)$.  In particular, $J^nK$ has no rational eigenslopes.
\end{lem}

%\begin{lem}

%Suppose that $K$ and $L$ are matrices in $SL_2(\mathbb R)$ such that $|trace(K)| > 2$. Then there exists a positive integer $N$ such that for $n \geq N$, $|trace(K^nL)| \neq 2$.
%\end{lem}

\begin{proof}
Let $t_n$ be the trace of $J^nK$. Since $|trace(J)| > 2$, $J$ is conjugate to a matrix $W$ in $SL_2(\mathbb R)$ of the form:

$$W = \left[ \begin{array}{cc} w & 0 \\ 0 & \frac{1}{w}\end{array}\right].$$

Without loss of generality, assume $w >1$ (the cases that $w <1$ or $w$ is negative are similar). Assume that $C$ is the conjugation matrix for $J$, i.e.~$C^{-1} J C = W$. Now set $Z = C^{-1} K C$. Thus $Z$ is a fixed matrix in $SL_2(\mathbb R)$:

$$Z = \left[ \begin{array}{cc} z_1 & z_2 \\ z_3 & z_4 \end{array}\right].$$

Since trace is invariant under conjugation, we have

$$t_n = trace(J^nK) = trace(C^{-1} J^n C \cdot C^{-1} K C)$$

$$= trace(C^{-1} J C \cdots C^{-1} J C \cdot C^{-1} K C) = trace(W^n Z).$$

\noindent Now, since $$W^n = \left[ \begin{array}{cc} w^n & 0 \\ 0 & \frac{1}{w^n}\end{array}\right],$$ it follows that $$t_n = trace(W^nZ) = trace\left( \left[ \begin{array}{cc} w^n & 0 \\ 0 & \frac{1}{w^n}\end{array}\right] \left[ \begin{array}{cc} z_1 & z_2 \\ z_3 & z_4 \end{array}\right]\right)$$ $$= trace\left( \left[ \begin{array}{cc} z_1w^n & z_2w^n \\ \frac{z_3}{w^n} & \frac{z_4}{w^n} \end{array}\right]\right) = z_1w^n + \frac{z_4}{w^n}.$$
\medskip

Let $k$ be a positive constant. First, assume that $z_1 = 0$. If $z_4 = 0$ as well, then $t_n = 0$ for any $n$; otherwise a high enough power of $N$ gives that for $n \geq N$, $|t_n| = \left| \frac{z_4}{w^n}\right| < 2k$. In either case, the trace is certainly less than $2k$.

Now assume that $z_1 \neq 0$. Then clearly, for large enough $n$, $|t_n| = \left|z_1w^n + \frac{z_4}{w^n}\right|$ can be taken to be greater than $2k$. The lemma follows by letting $k = 1/d(J^nK)$ in the former case and $k = d(J^nK)$ in the latter, and finally applying Lemma~\ref{l:eigenslope}. 

%$$z_1w^{2n} + \mp2 w^n + z_4 = 0,$$

%\noindent and thus applying the quadratic formula yields:

%$$ w^n = \frac{\pm2 \pm \sqrt{4 - 4z_1z_4}}{2z_1} = \pm \frac{1}{z_1} \pm \frac{\sqrt{1-z_1z_4}}{z_1},$$

%\noindent Since $w$ is assumed to be greater than 1, we can find a high enough power $N$ so that for $n \geq N$, the quantity $w^n$ is greater than all of these constants, thus showing that $t_n \neq \pm 2$ for these $n$.
\end{proof}

%Note that if we assume that both of the diagonal elements $z_1$ and $z_4$ in the matrix $Z$ are not equal to zero, then the proof of the lemma implies that there exists a positive integer $N$ such that either $|trace(K^nL)| > 2$ or $|trace(K^{-n} L)| > 2 $ for all $n \geq N$.
%\bigskip

We now prove Theorem~\ref{main_thm2}.

\begin{proof}

Let $M$ be a 3-manifold and let $T_1$ and $T_2$ be two torus components of $\partial M$. Let $\psi \colon T_1 \to T_2$ and $\sigma \colon T_2 \to T_2$ be homeomorphisms. We can use these maps to glue $M$ along $T_1$ and $T_2$ and obtain the manifold $M_{\sigma^n \psi}$.  Assume that $\sigma$ is Anosov.

For each $i = 1, 2$, fix a homology basis for $T_i$ as in Section 2. As homeomorphisms induce maps between homology, for convenience we can identify the homeomorphisms $\psi$ and $\sigma$ with their induced maps so that $\psi \colon H_1(T_1) \to H_1(T_2)$ and $\sigma \colon H_1(T_2) \to H_1(T_2)$ are regarded as elements of $SL_2(\mathbb Z)$, where $\sigma$ is hyperbolic (in particular, $|trace(\sigma) | > 2$). %Assume for convenience that the bases are chosen so that $\psi$ preserves basis vectors between $H_1(T_1)$ and $H_1(T_2)$.

Fix a positive integer $c$, and let $\mathcal C$ be a typed compatibility class of complexity at most $c$. Let $\Phi_{\mathcal C}$ be the element in $SL_2(\mathbb Q)$ given by Lemma~\ref{SL2Qlemma}. The element $\sigma^n \psi \Phi_{\mathcal C}$ is thus in $SL_2(\mathbb Q)$, and fixes a slope on $T_2$ if and only if $\sigma^n \psi \Phi_{\mathcal C}$ has an eigenvector $\vec v$ in $H_1(T_2)$. In other words, $\sigma^n \psi \Phi_{\mathcal C}$ has distance zero if and only if $\sigma^n \psi \Phi_{\mathcal C}$ has a rational eigenslope. 
%It is a straightforward exercise to show that a matrix in $SL_2(\mathbb R)$ fixes a vector if and only if its trace is 2. Thus the element $\sigma^n \psi \Phi_{\mathcal C}$ has a fixed point (mod $\pm$) if and only if its trace is $\pm 2$.

By taking $J = \sigma$ and $K = \psi \Phi_{\mathcal C}$ and applying Lemma~\ref{fixed_point_lem}, we obtain an integer $N_{\mathcal C}$ such that for all $n \geq N_{\mathcal C}$, $\sigma^n \psi \Phi_{\mathcal C}$ has no rational eigenslopes. Hence, $\sigma^n \psi \Phi_{\mathcal C}$ cannot have distance zero. Following Remark~\ref{r:finite}, there are a finite number of typed compatibility classes of complexity at most $c$. By taking $N$ to be the maximum over all integers $N_{\mathcal C}$ obtained as above, we have that $n \geq N$ implies that $\sigma^n \psi$ cannot have $c$-distance zero.

\end{proof}

\bibliographystyle{alpha}
\bibliography{Toroidal_self_gluings}

\end{document}